\documentclass[12pt,reqno]{amsart}

\usepackage[utf8]{inputenc}
\usepackage{amssymb,amsmath,amsthm}
\usepackage[margin=1in]{geometry}
\usepackage{hyperref}
\usepackage[normalem]{ulem}
\usepackage[makeroom]{cancel}
\newtheorem{lem}{Lemma}
\newtheorem{prop}[lem]{Proposition}
\newtheorem{thm}[lem]{Theorem}
\newtheorem{cor}[lem]{Corollary}
\newtheorem{conj}[lem]{Conjecture}
\newtheorem{rem}[lem]{Remark}

\def\Ex{{\mathbb E}}
\def\Pr{{\mathbb P}}
\def\er{{\mathbb R}}
\def\ve{\varepsilon}
\def\diam{\mathrm{diam}}

\title{Hanson-Wright Inequality in Banach Spaces}
\author{Rados{\l}aw Adamczak}
\author{Rafa{\l} Lata{\l}a}
\author{Rafa{\l} Meller}
\address[R.A.]{Institute of Mathematics of the Polish Academy of Sciences, \'Sniadeckich 8, 00-656 Warsaw, Poland}
\address[R.A., R.L., R.M.]{Institute of Mathematics, University of Warsaw, Banacha 2, 02--097 Warsaw, Poland.}
\email{r.adamczak@mimuw.edu.pl, r.latala@mimuw.edu.pl, r.meller@mimuw.edu.pl}
\thanks{R.A.\ was supported by National Science Centre Poland grant 2015/18/E/ST1/00214,
R.L.\ and R.M.\ were supported by the National Science Centre Poland grant 2015/18/A/ST1/00553.
This work was initiated in the Fall of 2017, while R.L. was in residence at the Mathematical Sciences Research Institute in Berkeley, California, supported by NSF grant DMS-1440140.}

\keywords{tail and moment inequalities, quadratic forms, Hanson-Wright inequality, Gaussian chaoses, Gaussian processes, metric entropy}
\subjclass[2010]{Primary: 60E15, Secondary: 60G15, 60B11}
\begin{document}

\maketitle

\begin{abstract}
We discuss two-sided bounds for moments and tails of quadratic forms in Gaussian random variables
with values in Banach spaces. We state a natural conjecture and show that it holds up to additional logarithmic factors.
Moreover in a certain class of Banach spaces (including $L_r$-spaces) these logarithmic factors may be eliminated. As a
corollary we derive upper bounds for tails and moments of quadratic forms in subgaussian random variables, which
extend the Hanson-Wright inequality.

\end{abstract}

\section{Introduction and main results}

The Hanson-Wright inequality gives an upper bound for tails of real quadratic forms in independent subgaussian random variables. Recall that a random variable $X$ is called $\alpha$-\emph{subgaussian} if for every $t > 0$, $\Pr(|X| \ge t) \le 2\exp(-t^2/2\alpha^2)$. The Hanson-Wright inequality states
that for any sequence of independent mean zero $\alpha$-subgaussian random variables $X_1,\ldots,X_n$ and any symmetric matrix $A=(a_{ij})_{i,j\leq n}$ one has
\begin{equation}
\label{eq:HW}
\Pr\left(\left|\sum_{i,j=1}^n a_{ij}(X_iX_j-\Ex (X_iX_j))\right|\geq t\right)
\leq 2\exp\left(-\frac{1}{C}\min\left\{\frac{t^2}{\alpha^4\|A\|_{\mathrm{HS}}},\frac{t}{\alpha^2\|A\|_{\mathrm{op}}}\right\}\right),
\end{equation}
where in the whole article we use the letter $C$ to denote universal constants  which may differ at each occurrence. Estimate \eqref{eq:HW}
was essentially established in \cite{HW} in the symmetric and in \cite{W} in the mean zero case (in fact in both papers the operator norm of $A$ was replaced by the operator norm of $(|a_{ij}|)_{i,j\le n}$, which in general could be much bigger,
proofs of \eqref{eq:HW} may be found in \cite{BM} and \cite{RV}).

The Hanson-Wright inequality has found numerous applications in high-dimensional probability and statistics, as well as in random matrix theory (see e.g., \cite{vershynin_2018}). However in many problems one faces the need to analyze not a single quadratic form but a supremum of a collection of them or equivalently a norm of a quadratic form with coefficients in a Banach space. Let us mention a few recent developments, relying on the analysis of suprema of quadratic forms in independent random variables. For instance in \cite{KMR} estimates for such suprema are applied to the theory of compressed sensing with circulant  type matrices. In \cite{A} they are used to give an alternate proof of the result from \cite{KL} concerning concentration properties for sample covariance operators corresponding to Banach space-valued Gaussian random variables. In \cite{DE} one can find applications to $M$-estimation in the context of linear random-effects models.

The inequalities concerning tails of Banach space valued quadratic forms available in the literature (see e.g., ineq. \eqref{eq:Borel-Arcones-Gine} below) are usually expressed in terms of quantities which themselves are troublesome to analyze. The main objective of this article is to provide estimates on vector-valued quadratic forms which can be applied more easily and are of optimal form.

The main step in modern proofs  of the Hanson-Wright inequality is to get a bound similar to \eqref{eq:HW} in the Gaussian case. The extension to general subgaussian variables is then obtained with use of the by now standard tools of probability in Banach spaces, such as decoupling, symmetrization and the contraction principle. Via Chebyshev's
inequality to obtain a tail estimate it is enough to bound appropriately the moments of quadratic forms
in the case when $X_i=g_i$ are standard Gaussian $\mathcal{N}(0,1)$ random variables.
One may in fact show that (cf.\ \cite{LaSM,LaAoP})
\begin{equation}
\left(\Ex\left|\sum_{i,j=1}^n a_{ij}(g_ig_j-\delta_{ij})\right|^p\right)^{1/p}
\sim p\|A\|_{\mathrm{op}}+\sqrt{p}\|A\|_{\mathrm{HS}}, \label{mom}
\end{equation}
where $\delta_{ij}$ is the Kronecker delta, and $\sim$ stands for a comparison up to universal multiplicative constants.

Following the same line of arguments, in order to extend the Hanson-Wright bound to the Banach space setting we first estimate moments of centered vector-valued Gaussian quadratic forms, i.e. quantities
\[
\left\|\sum_{i,j=1}^n a_{ij}(g_ig_j-\delta_{ij})\right\|_p=\left(\Ex\left\|\sum_{i,j=1}^n a_{ij}
(g_ig_j-\delta_{ij})\right\|^p\right)^{1/p}, \quad p\geq 1,
\]
where $A=(a_{ij})_{i,j\leq n}$ is a symmetric matrix with values in a normed space $(F,\|\ \|)$.
We note that (as mentioned above) there exist two-sided estimates for the moments of Gaussian quadratic forms with vector-valued coefficients. To the best of our knowledge they were obtained first in \cite{Bo} and then they were reproved in various context by several authors (see e.g., \cite{AG,Le,LT}).
They state that for $p\ge 1$,
\begin{align}
\notag
\left\|\sum_{ij}a_{ij}(g_ig_j-\delta_{ij})\right\|_p \sim&\ \Ex \left\|\sum_{ij}a_{ij}(g_ig_j-\delta_{ij})\right\| + \sqrt{p}\Ex \sup_{\|x\|_2\le 1} \left \| \sum_{ij} a_{ij}x_i g_j\right\|
\\
\label{eq:Borel-Arcones-Gine}
&+ p\sup_{\|x\|_2\le 1,\|y\|_2\le 1} \left\| \sum_{ij}a_{ij} x_i y_j\right\|.
\end{align}

Unfortunately the second term on the right hand side of \eqref{eq:Borel-Arcones-Gine} is usually difficult to estimate. The main effort in this article will be to replace it by quantities which even if still involve expected values of Banach space valued random variables in many situations can be handled more easily. More precisely, we will obtain inequalities in which additional suprema over Euclidean spheres are placed outside the expectations, which reduces the complexity of the involved stochastic processes. As one of the consequences we will derive two-sided bounds in $L_r$ spaces involving only purely deterministic quantities.

\medskip

Our first observation is a simple lower bound

\begin{prop}
\label{prop:lower2d}
Let $(a_{ij})_{i,j \leq n}$ be a symmetric matrix with values in a normed space \mbox{$(F,\| \ \cdot \|)$}. Then for any $p\geq 1$ we have
\begin{align*}
\left\|\sum_{ij}a_{ij}(g_ig_j-\delta_{ij})\right\|_p
&\geq
\frac{1}{C}\Bigg(\Ex\left\|\sum_{ij}a_{ij}(g_ig_j-\delta_{ij})\right\|
+\sqrt{p}\sup_{\|x\|_2\leq 1}\Ex\left\|\sum_{i\neq j}a_{ij}x_ig_j\right\|
\\
&\phantom{aaaaa}+\sqrt{p}\sup_{\|(x_{ij})\|_2\leq 1}\left\|\sum_{ij}a_{ij}x_{ij}\right\|
+p\sup_{\|x\|_2\leq 1,\|y\|_2\leq 1}\left\|\sum_{ij}a_{ij}x_iy_j\right\|
\Bigg).
\end{align*}
\end{prop}

This motivates the following conjecture.

\begin{conj}
\label{conj1_2d}
 Under the assumptions of Proposition \ref{prop:lower2d} we have
\begin{align*}
\left\|\sum_{ij}a_{ij}(g_ig_j-\delta_{ij})\right\|_p
&\leq
C\Bigg(\Ex\left\|\sum_{ij}a_{ij}(g_ig_j-\delta_{ij})\right\|
+\sqrt{p}\sup_{\|x\|_2\leq 1}\Ex\left\|\sum_{i\neq j}a_{ij}x_ig_j\right\|
\\
&\phantom{aaaaa}+\sqrt{p}\sup_{\|(x_{ij})\|_2\leq 1}\left\|\sum_{ij}a_{ij}x_{ij}\right\|
+p\sup_{\|x\|_2\leq 1,\|y\|_2\leq 1}\left\|\sum_{ij}a_{ij}x_iy_j\right\|
\Bigg).
\end{align*}
\end{conj}

We are able to show that the conjectured estimate holds up to logarithmic factors.

\begin{thm}
\label{thm:uppper2d1}
Let $(a_{ij})_{i,j \leq n}$ be a symmetric matrix with values in a normed space $(F,\| \ \cdot \|)$. Then for any $p\geq 1$ the following two estimates hold
\begin{align}
\notag
\left\|\sum_{ij}a_{ij}(g_ig_j-\delta_{ij})\right\|_p
&\leq
C\Bigg(\log(ep)\Ex\left\|\sum_{ij}a_{ij}(g_ig_j-\delta_{ij})\right\|
+\sqrt{p}\sup_{\|x\|_2\leq 1}\Ex\left\|\sum_{i\neq j}a_{ij}x_ig_j\right\|
\\
\label{eq:upperest1}
+\sqrt{p}&\sup_{\|(x_{ij})\|_2\leq 1}\left\|\sum_{ij}a_{ij}x_{ij}\right\|
+p\log(ep)\sup_{\|x\|_2\leq 1,\|y\|_2\leq 1}\left\|\sum_{ij}a_{ij}x_iy_j\right\|
\Bigg)
\end{align}
and
\begin{align}
\notag
\left\|\sum_{ij}a_{ij}(g_ig_j-\delta_{ij})\right\|_p
&\leq
C\Bigg(\Ex\left\|\sum_{ij}a_{ij}(g_ig_j-\delta_{ij})\right\|
+\sqrt{p}\sup_{\|x\|_2\leq 1}\Ex\left\|\sum_{i\neq j}a_{ij}x_ig_j\right\|
\\
\label{eq:upperest2}
+\sqrt{p}&\log(ep)\sup_{\|(x_{ij})\|_2\leq 1}\left\|\sum_{ij}a_{ij}x_{ij}\right\|
+p\sup_{\|x\|_2\leq 1,\|y\|_2\leq 1}\left\|\sum_{ij}a_{ij}x_iy_j\right\|
\Bigg).
\end{align}
\end{thm}

One of the main reasons behind the appearance of additional logarithmic factors is lack of good Sudakov-type estimates for
Gaussian quadratic forms. Such bounds hold for linear forms and as a result we may show the following
($(g_{ij})_{i,j\leq n}$ below denote as usual i.i.d. ${\mathcal N}(0,1)$ random variables).

\begin{thm}
\label{thm:upper2d2}
Under the assumptions of Theorem \ref{thm:uppper2d1} we have
\begin{align}
\notag
\left\|\sum_{ij}a_{ij}(g_ig_j-\delta_{ij})\right\|_p
\leq
C\Bigg(&\Ex\left\|\sum_{ij}a_{ij}(g_ig_j-\delta_{ij})\right\|+\Ex\left\|\sum_{i\neq j}a_{ij}g_{ij}\right\|
\\
\notag
&+\sqrt{p}\sup_{\|x\|_2\leq 1}\Ex\left\|\sum_{i\neq j}a_{ij}x_ig_j\right\|
+\sqrt{p}\sup_{\|(x_{ij})\|_2\leq 1}\left\|\sum_{ij}a_{ij}x_{ij}\right\|
\\
\label{eq:upperest3}
&+p\sup_{\|x\|_2\leq 1,\|y\|_2\leq 1}\left\|\sum_{ij}a_{ij}x_iy_j\right\|
\Bigg).
\end{align}
\end{thm}

In particular we know that Conjecture \ref{conj1_2d} holds in Banach spaces, in which Gaussian quadratic forms dominate in mean Gaussian linear forms, i.e. in Banach spaces $(F,\|\ \|)$ for which there exists a constant $\tau<\infty$ such
for any finite symmetric matrix $(a_{ij})_{i,j\leq n}$ with values in $F$ one has
\begin{equation}
\Ex\left\|\sum_{i\neq j}a_{ij}g_{ij}\right\|\leq \tau \Ex\left\|\sum_{ i\neq j}a_{ij}g_ig_j\right\|. \label{war}
\end{equation}

This condition has ben introduced in \cite{vNW} as a one sided version of Pisier's contraction property (see \cite{Pisier1}) in the context of multiple stochastic integration in the vector valued setting. It is easy to check (see Proposition \ref{prop:estLr} below) that it holds for $L_r$-spaces with $\tau=\tau(r)$, but the class of spaces satisfying it is much larger. For instance \eqref{war} holds in all type 2 spaces and for Banach lattices it is equivalent to finite cotype. We refer the reader to \cite[Chapter 7]{Hytonen} for a more detailed presentation.

\begin{rem}
Let us consider a non-centered  Gaussian quadratic form $S=\sum_{i,j}a_{ij}g_ig_j$. By the triangle inequality for $p\ge 1$ one has
$\|S\|_p\le  \|\Ex S\|+\|S- \Ex S\|_p$, whereas Jensen's inequality implies $\|S\|_p \ge \Ex \|S\| \ge \|\Ex S\|$ and as a consequence $\|S - \Ex S\|_p \le 2 \|S\|_p$.
Thus $\|S\|_p\sim \|\Ex S\|+\|S- \Ex S\|_p$ and so Proposition \ref{prop:lower2d} yields
\begin{align*}
\left\|\sum_{ij}a_{ij}g_ig_j\right\|_p
&\geq
\frac{1}{C}\Bigg(\Ex\left\|\sum_{ij}a_{ij}g_ig_j\right\|
+\sqrt{p}\sup_{\|x\|_2\leq 1}\Ex\left\|\sum_{i\neq j}a_{ij}x_ig_j\right\|
\\
&\phantom{aaaaa}+\sqrt{p}\sup_{\|(x_{ij})\|_2\leq 1}\left\|\sum_{ij}a_{ij}x_{ij}\right\|
+p\sup_{\|x\|_2\leq 1,\|y\|_2\leq 1}\left\|\sum_{ij}a_{ij}x_iy_j\right\|
\Bigg).
\end{align*}
and Theorem \ref{thm:upper2d2} implies
\begin{align*}
\left\|\sum_{ij}a_{ij}g_ig_j\right\|_p
&\leq
C\Bigg(\Ex\left\|\sum_{ij}a_{ij}g_ig_j\right\|+\Ex\left\|\sum_{i\neq j}a_{ij}g_{ij}\right\|
+\sqrt{p}\sup_{\|x\|_2\leq 1}\Ex\left\|\sum_{i\neq j}a_{ij}x_ig_j\right\|
\\
&\phantom{aaaaa}+\sqrt{p}\sup_{\|(x_{ij})\|_2\leq 1}\left\|\sum_{ij}a_{ij}x_{ij}\right\|
+p\sup_{\|x\|_2\leq 1,\|y\|_2\leq 1}\left\|\sum_{ij}a_{ij}x_iy_j\right\|
\Bigg).
\end{align*}
\end{rem}

Proposition \ref{prop:lower2d} and Theorem \ref{thm:upper2d2} may be expressed in terms of tails.

\begin{thm}
\label{thm:tails2d}
Let $(a_{ij})_{i,j \leq n}$ be a symmetric matrix with values in a normed space $(F,\| \ \cdot \  \|)$.
Then for any $t>0$,
\begin{align}\label{eq:thm6-lower-bound}
\Pr\left(\left\|\sum_{ij}a_{ij}(g_ig_j-\delta_{ij})\right\|
\geq t+\frac{1}{C}\Ex\left\|\sum_{ij}a_{ij}(g_ig_j-\delta_{ij})\right\|\right)
\geq \frac{1}{C}\exp\left(-C\min\left\{\frac{t^2}{U^2},\frac{t}{V}\right\}\right),
\end{align}
where
\begin{align}
U&=\sup_{\|x\|_2\leq 1}\Ex\left\|\sum_{i \neq j}a_{ij}x_ig_j\right\|
+\sup_{\|(x_{ij})\|_2\leq 1}\left\|\sum_{ij}a_{ij}x_{ij}\right\|, \label{u}
\\
V&=\sup_{\|x\|_2\leq 1,\|y\|_2\leq 1}\left\|\sum_{ij}a_{ij}x_iy_j\right\|. \label{v}
\end{align}
Moreover, for
$t>C(\Ex\|\sum_{ij}a_{ij}(g_ig_j-\delta_{ij})\|+\Ex\|\sum_{i\neq j}a_{ij}g_{ij}\|)$ we have
\begin{align}\label{eq:thm6-upper-bound}
\Pr\left(\left\|\sum_{ij}a_{ij}(g_ig_j-\delta_{ij})\right\|\geq t\right)
\leq 2\exp\left(-\frac{1}{C}\min\left\{\frac{t^2}{U^2},\frac{t}{V}\right\}\right).
\end{align}
\end{thm}

As a corollary to Theorem \ref{thm:upper2d2} we get a Hanson-Wright-type inequality for Banach space valued quadratic forms in general independent subgaussian random variables. Its proof as well as proofs of Theorems \ref{thm:uppper2d1}, \ref{thm:upper2d2} and  \ref{thm:tails2d} is presented in Section \ref{sec:proofs-of-main-results}.

\begin{thm}\label{thm:hw}
Let $X_1,X_2,\ldots,X_n$ be independent mean zero $\alpha$-subgaussian random variables. Then for any symmetric  matrix
$(a_{ij})_{i,j\leq n}$ with values in a normed space $(F,\| \ \cdot \ \|)$ and
$t>C\alpha^2(\Ex\|\sum_{ij}a_{ij}(g_ig_j-\delta_{ij})\|+\Ex\|\sum_{i \neq j}a_{ij}g_{ij}\|)$ we have
\begin{equation}
\label{eq:HWnorm}
\Pr\left(\left\|\sum_{ij}a_{ij}(X_iX_j-\Ex(X_iX_j))\right\|\geq t\right)
\leq 2\exp\left(-\frac{1}{C}\min\left\{\frac{t^2}{\alpha^4 U^2},\frac{t}{\alpha^2 V}\right\}\right),
\end{equation}
where  $U$ and $V$ are as in Theorem \ref{thm:tails2d}.
\end{thm}

\begin{rem}
It is not hard to check that in the case $F=\er$ we have $U\sim \|(a_{ij})\|_{\mathrm{HS}}$ and
$V=\|(a_{ij})\|_{\mathrm{op}}$. Moreover,
\[
\Ex\left\|\sum_{ij}a_{ij}(g_ig_j-\delta_{ij})\right\|+\Ex\left\|\sum_{i  \neq j}a_{ij}g_{ij}\right\|
\leq (\sqrt{2}+1)\|(a_{ij})\|_{\mathrm{HS}},
\]
so  for $t<C'\alpha^2(\Ex\|\sum_{ij}a_{ij}(g_ig_j-\delta_{ij})\|+\Ex\|\sum_{i \neq j}a_{ij}g_{ij}\|)$ and sufficiently large $C$ the right hand side of \eqref{eq:HWnorm} is at least 1.
Hence \eqref{eq:HWnorm}
holds for any $t>0$ in the real case and is equivalent to the Hanson-Wright bound.
\end{rem}

\begin{rem}
Proposition \ref{prop:diagest} below shows that we may replace in all estimates above the term
$\sup_{\|x\|_2\leq 1}\Ex\|\sum_{i \neq j}a_{ij}x_ig_j\|$ by
$\sup_{\|x\|_2\leq 1}\Ex\|\sum_{ij}a_{ij}x_ig_j\|$.
\end{rem}

\begin{rem}
We are able to derive similar estimates as discussed in this paper for Banach space valued Gaussian chaoses of arbitrary degree. Formulas are however more complicated and the proof is more technical. For these reasons we decided to include details in a separate forthcoming paper \cite{ALM}.
\end{rem}

The organization of the paper is as follows. In the next section we discuss a few corollaries of Theorems \ref{thm:upper2d2} and \ref{thm:hw}. In Section 3 we prove Proposition \ref{prop:lower2d} and show that it is enough
to  bound separately moments of diagonal and off-diagonal parts of chaoses.
In Section 4 we reduce Theorems \ref{thm:uppper2d1} and \ref{thm:upper2d2} to the problem of estimating means of suprema of certain Gaussian processes. In Section 5 we show how to bound expectations of such suprema -- the main new ingredient are entropy bounds presented in Corollary \ref{cor:entrestch} (derived via volumetric-type arguments).
Unfortunately our entropy bounds are too weak to use the Dudley integral bound. Instead, we present a technical chaining argument (of similar type as in \cite{LaAoP}). In the last section we conclude the proofs of main Theorems.

\section{Consequences and extensions}

\subsection{$L_r$-spaces}

We will now restrict our attention to $L_r$ spaces for $1 \le r<\infty$. It turns out that in this case one may express all parameters of Theorem \ref{thm:upper2d2} without any expectations as is shown in the proposition below. As a corollary we will obtain a version of the Hanson-Wright inequality in terms of only deterministic quantities, in particular without the additional term $\Ex\|\sum_{i \neq j}a_{ij}g_{ij}\|$, which shows that Conjecture \ref{conj1_2d} is satisfied in these spaces (with constants depending on $r$). We note that the following proposition implies also that $L_r$ spaces satisfy \eqref{war} with $\tau=Cr$. In fact, using the general theory presented in \cite{Hytonen}, one can improve this to $\tau=C\sqrt{r}$, however we are not going to use directly any of these bounds on $\tau$.

\begin{prop}
\label{prop:estLr}
For any symmetric matrix $(a_{ij})_{i,j\leq n}$ with values in $L_r=L_r(X,\mu)$, $1\leq r<\infty$
and $x_1,\ldots,x_n\in \er$ we have
\begin{align}
&\frac{1}{C}\left\|\sqrt{\sum_{ij}a_{ij}^2}\right\|_{L_r}
\leq \Ex\left\|\sum_{ij}a_{ij}g_{ij}\right\|_{L_r}
\leq C\sqrt{r}\left\|\sqrt{\sum_{ij}a_{ij}^2}\right\|_{L_r}, \label{loc1}
\\
&\frac{1}{C}\left\|\sqrt{\sum_{j}\left(\sum_i a_{ij}x_i\right)^2}\right\|_{L_r}
\leq \Ex\left\|\sum_{ij}a_{ij}x_{i}g_j\right\|_{L_r}
\leq C\sqrt{r}\left\|\sqrt{\sum_{j}\left(\sum_i a_{ij}x_i\right)^2}\right\|_{L_r}, \label{loc2}
\\
&\frac{1}{C\sqrt{r}}\left\|\sqrt{\sum_{ij}a_{ij}^2}\right\|_{L_r}
\leq \Ex\left\|\sum_{ij}a_{ij}(g_{i}g_j-\delta_{ij})\right\|_{L_r}
\leq Cr\left\|\sqrt{\sum_{ij}a_{ij}^2}\right\|_{L_r}. \label{loc3}
\end{align}
\end{prop}

\begin{proof}
For any $a_i$'s in $L_r$ the Gaussian concentration yields
\begin{align*}
\Ex \left\| \sum_i a_i g_i \right\|_{L_r} \leq \left(\Ex \left\| \sum_i a_i g_i \right\|^r_{L_r} \right)^{1/r} \leq C\sqrt{r}\Ex \left\| \sum_i a_i g_i \right\|_{L_r}. 
\end{align*}
Since
\begin{align*}
\left(\Ex \left\| \sum_i a_i g_i \right\|^r_{L_r} \right)^{1/r}&=\left(\int_X \Ex \left| \sum_i a_i (x) g_i \right|^r d \mu(x) \right)^{1/r}\\
&=\left(\int_X \Ex|g_1|^r  \Big(\sum_i a^2_i(x)\Big)^{r/2} d \mu(x) \right)^{1/r}
\sim \sqrt{r} \left\| \sqrt{\sum_i a^2_i}\right\|_{L_r},
\end{align*}
estimates \eqref{loc1},\eqref{loc2} follow easily.
The proof of \eqref{loc3} is analogous.  It is enough to observe that \eqref{eq:Borel-Arcones-Gine} yields that
\begin{align*}
\Ex\left\|\sum_{ij}a_{ij}(g_{i}g_j-\delta_{ij})\right\|_{L_r}
\leq \left( \Ex\left\|\sum_{ij}a_{ij}(g_{i}g_j-\delta_{ij})\right\|_{L_r}^r\right)^{1/r}
\leq C r \Ex\left\|\sum_{ij}a_{ij}(g_{i}g_j-\delta_{ij})\right\|_{L_r}
\end{align*}
and \eqref{mom} implies for any $x \in X$,
$$\frac{\sqrt{r}}{C}  \sqrt{\sum_{ij} a^2_{ij}(x) }
\leq  \left( \Ex \left|\sum_{ij}a_{ij}(x)(g_{i}g_j-\delta_{ij}) \right| ^{r} \right)^{1/r}
\leq  C r   \sqrt{\sum_{ij} a^2_{ij}(x)}.$$
\end{proof}

The above proposition, together with Proposition \ref{prop:lower2d} and Theorem \ref{thm:upper2d2} immediately yield the following corollary (in particular they imply that Conjecture \ref{conj1_2d} holds in $L_r$ spaces with $r$-dependent constants). Below $\sim^r$ denotes comparison up to constants depending only on $r$.

\begin{cor}
For any symmetric matrix $(a_{ij})_{ij}$ with values in $L_r$ and  $p\geq 1$ we have
\begin{align*}
\left\|\sum_{ij}a_{ij}(g_ig_j-\delta_{ij})\right\|_p
\sim^r&
\left\|\sqrt{\sum_{ij}a_{ij}^2}\right\|_{L_r}
+\sqrt{p} \sup_{\|x\|_2\leq 1}\left\|\sqrt{\sum_{j}\left(\sum_{ i\neq j} a_{ij}x_i\right)^2}\right\|_{L_r}
\\
&+\sqrt{p}\sup_{\|(x_{ij})\|_2\leq 1}\left\|\sum_{ij}a_{ij}x_{ij}\right\|_{L_r}
+p\sup_{\|x\|_2\leq 1,\|y\|_2\leq 1}\left\|\sum_{ij}a_{ij}x_iy_j\right\|_{ L_r}.
\end{align*}
The implicit constants in the estimates for moments can be taken to be equal to $Cr$ in the upper bound and
$r^{-1/2}/C$ in the lower bound.
\end{cor}

Theorem \ref{thm:hw} and Proposition \ref{prop:estLr} imply the following Hanson-Wright-type estimate in $L_r$-spaces.

\begin{cor}
Let $X_1,X_2,\ldots,X_n$ be independent mean zero $\alpha$-subgaussian random variables. Then for any symmetric finite matrix
$(a_{ij})_{i,j\leq n}$ with values in $L_r=L_r(X,\mu)$, $1\leq r<\infty$ and
$t>C\alpha^2r\|\sqrt{\sum_{ij}a_{ij}^2}\|_{L_r}$ we have
\begin{equation}
\label{eq:HWnorm1}
\Pr\left(\left\|\sum_{ij}a_{ij}(X_iX_j-\Ex(X_iX_j))\right\|_{L_r}\geq t\right)
\leq 2\exp\left(-\frac{1}{C}\min\left\{\frac{t^2}{\alpha^4 r U^2},\frac{t}{\alpha^2 V }\right\}\right),
\end{equation}
where
\begin{align*}
U&=\sup_{\|x\|_2\leq 1}\left\|\sqrt{\sum_{j}\left(\sum_{ i\neq j }a_{ij}x_i\right)^2}\right\|_{L_r}
+\sup_{\|(x_{ij})\|_2\leq 1}\left\|\sum_{ij}a_{ij}x_{ij}\right\|_{L_r},
\\
V&=\sup_{\|x\|_2\leq 1,\|y\|_2\leq 1}\left\|\sum_{ij}a_{ij}x_iy_j\right\|_{L_r}.
\end{align*}
\end{cor}

\subsection{Spaces of type 2}\
Recall that  a normed space $F$ is of type 2 with constant $\lambda$ if for every positive integer $n$ and $v_1,\ldots,v_n \in F$,
\begin{equation}\label{eq:type-2-condition}
  \Ex\left\|\sum_{i=1}^n v_i \varepsilon_i\right\| \le \lambda\sqrt{\sum_{i=1}^n \|v_i\|^2},
\end{equation}
where $\varepsilon_1,\varepsilon_2,\ldots$ is a sequence of independent Rademacher variables.
The importance of this class of spaces stems from the fact that if $F$ is of type 2 then one can easily bound from above expectations of linear combinations or more generally multilinear forms in independent centered random variables with $F$-valued coefficients. Let us now introduce the basic (by now classical) tools allowing to prove such estimates and show how they can be combined with Theorem \ref{thm:hw} to yield a tail inequality for quadratic forms with the right hand side expressed in terms of deterministic quantities (i.e., not involving expectations of Banach space valued random variables).

The first tool we are going to use is the following symmetrization inequality (see, e.g., \cite[Lemma 6.3]{LT} for a more general formulation).

\begin{lem}\label{le:symmetrization}
Let $Y_1,\ldots, Y_n$ be independent random variables with values in a Banach space $F$. Let moreover $\varepsilon_1,\ldots,\varepsilon_n$ be i.i.d. Rademacher variables independent of $(Y_i)_{i=1}^n$. Then for any $p\ge 1$,
\begin{displaymath}
  \Big\|\sum_{i=1}^n (Y_i-\Ex Y_i)\Big\|_p \le 2\Big\|\sum_{i=1}^n \varepsilon_i Y_i\Big\|_p.
\end{displaymath}
\end{lem}

The other fact we need is a decoupling inequality for quadratic forms. For symmetric random variables it has been proved in \cite{K}, the general case can be immediately obtained from a decoupling result for $U$-statistics, proved in \cite{dlPMS}.

\begin{thm}
\label{th:decoupling}
Let $X_1,\ldots,X_n$ be a sequence of independent random variables and let $X_1',\ldots,$ $X_n'$ be its independent
copy. Then for any symmetric matrix $(a_{ij})_{i,j\le n}$ with coefficients in a normed space $(F,\|\cdot\|)$
and any $p\ge 1$,
\begin{equation}
\label{eq:decoupling}
\frac{1}{C}\Big\|\sum_{i\neq j} a_{ij}X_i X_j'\Big\|_p \le \Big\|\sum_{i\neq j} a_{ij}X_i X_j\Big\|_p \le C\Big\|\sum_{i\neq j} a_{ij}X_i X_j'\Big\|_p.
\end{equation}
\end{thm}

The advantage of working with bilinear forms in two independent sequences of random variables instead of quadratic forms in a single sequence is that one can view the former conditionally as a sum of independent random variables, which allows to use a variety of classical tools.

Let us now return to type 2 spaces and show how the above lemmas can be combined with condition \eqref{eq:type-2-condition} in order to bound the expectations appearing in Theorem \ref{thm:hw}.

First, by Lemma \ref{le:symmetrization} and \eqref{eq:type-2-condition} one easily obtains that if $F$ is of type 2 with constant $\lambda$ then for any independent random variables $X_i$,
\begin{displaymath}
\Ex \left \| \sum_i a_i (X_i^2 - \Ex X_i^2)\right\| \le 2\Ex \left\|\sum_i a_i \varepsilon_i X_i^2\right\|
\le 2\lambda \Ex \sqrt{\sum_i \|a_i\|^2 X_i^4}\le  2\lambda \sqrt{\sum_i \|a_i\|^2 \Ex X_i^4},
\end{displaymath}
where $\varepsilon_1,\ldots,\varepsilon_n$ are independent Rademacher variables, independent of $X_1,\ldots,X_n$.

Moreover, if $\Ex X_i=0$, then using the notation from Theorem \ref{th:decoupling} and denoting now by $\varepsilon_1,\varepsilon_1',\ldots,\varepsilon_n,\varepsilon_n'$ independent Rademacher variables, independent of $X_1,X_1',\ldots,X_n,X_n'$, we obtain
\begin{align*}
\left(\Ex\left \|\sum_{i\neq j} a_{ij}X_i X_j\right\|\right)^2
&\le \Ex \left\| \sum_{i\neq j} a_{ij}X_i X_j\right\|^2
\le C\Ex \left\| \sum_{i\neq j} a_{ij}X_i X_j'\right\|^2
\\
&\le
16 C\Ex \left\| \sum_{i\neq j} a_{ij}\varepsilon_i \varepsilon_j' X_i X_j'\right\|^2
\leq 32 C\lambda^2\sum_{i}\Ex(X_i^2)\Ex\left\|\sum_{j\colon j\neq i}a_{ij}\varepsilon_j'X_j'\right\|^2
\\
&\leq 64C\lambda^4\sum_{i\neq j}\|a_{ij}\|^2\Ex X_i^2 \Ex X_j^2.
\end{align*}
The second inequality above is an application of Theorem \ref{th:decoupling}, the third one follows from an iterated conditional application of Lemma \ref{le:symmetrization} and the last two by a conditional application of \eqref{eq:type-2-condition} together with the Khintchine-Kahane inequality $\|\sum_{i}v_i\ve_i\|_2\leq \sqrt{2}\|\sum_{i}v_i\ve_i\|_1$ (\cite{Ka}, see \cite{LaO} for the optimal constant $\sqrt{2}$).

 \medskip

Combining the above estimates with Theorem \ref{thm:hw} yields immediately the following

\begin{cor}
Let $X_1,X_2,\ldots,X_n$ be independent mean zero $\alpha$-subgaussian random variables and let $F$ be a normed space of type two constant $\lambda$. Then for any symmetric finite matrix
$(a_{ij})_{i,j\leq n}$ with values in $F$ and
$t>C\lambda^2 \alpha^2 \sqrt{\sum_{ij}\|a_{ij}\|^2}$ we have
\begin{equation}
\Pr\left(\left\|\sum_{ij}a_{ij}(X_iX_j-\Ex(X_iX_j))\right\|\geq t\right)
\leq 2\exp\left(-\frac{1}{C}\min\left\{\frac{t^2}{\alpha^4U^2},\frac{t}{\alpha^2V}\right\}\right),
\end{equation}
where
\begin{align*}
U&=\lambda \sup_{\|x\|_2\leq 1}\sqrt{\sum_{j}\Big\|\sum_{i\neq j} a_{ij}x_i\Big\|^2}
+\sup_{\|(x_{ij})\|_2\leq 1}\left\|\sum_{ij}a_{ij}x_{ij}\right\|,
\\
V&=\sup_{\|x\|_2\leq 1,\|y\|_2\leq 1}\left\|\sum_{ij}a_{ij}x_iy_j\right\|.
\end{align*}
\end{cor}

\begin{rem}
As already mentioned in the introduction, type 2 spaces satisfy \eqref{war} with $\tau\leq \lambda$
(cf. \cite[Theorem 7.1.20]{Hytonen}) and thus also Conjecture \ref{conj1_2d} with a constant depending on $\lambda$ .
\end{rem}

\begin{rem}
We note that from Theorem \ref{thm:hw} one can also derive similar inequalities for suprema of quadratic forms over VC-type classes of functions appearing e.g., in the analysis of randomized $U$-processes (cf. e.g., \cite[Chapter 5.4]{delaPenaGine}).
\end{rem}

\subsection{Random vectors with dependencies}\

Let us now consider a random vector $X= (X_1,\ldots,X_n)$, possibly with dependent coordinates, and assume that there exists an $\alpha$-Lipschitz map $\psi\colon \er^n \to \er^n$ such that  $X$ has the same distribution as $\psi(G)$, where $G$ is a standard Gaussian vector in $\er^n$. An important class of vectors $X$ with this property is provided by the celebrated Caffarelli contraction principle \cite{Caffarelli_contraction}, which asserts that the map $\psi$ exists if $X$ has density of the form $e^{-V}$ with $\nabla^2 V \ge \alpha^{-2} {\rm Id}$.
As observed by Ledoux and Oleszkiewicz \cite[Corollary 1]{LO}, by combining the well known comparison result due to Pisier \cite{Pisier} with a stochastic domination-type argument, one gets that for any smooth function $f\colon \er^n \to F$, and any $p\ge 1$,
\begin{align}\label{eq:Pisier}
  \|f(X) - \Ex f(X)\|_p \le \frac{\pi\alpha}{2}\|\langle \nabla f(X), G\rangle\|_p,
\end{align}
where here and subsequently $G_n$ is a standard Gaussian vector in $\er^n$  independent of $X$ and for $a \in F^n$, $b \in \er^n$ we denote $\langle a,b\rangle = \sum_{i=1}^n a_i b_i$.
This inequality together with Theorem \ref{thm:upper2d2} allow us to implement a simple argument from \cite{AW} and obtain inequalities for quadratic forms and more general $F$-valued functions of the random vector $X$. Below, we will denote the second partial derivatives of $f$ by $\partial_{ij} f$.  For the sake of brevity, we will focus on moment estimates, clearly tail bounds follow from them by an application of the Chebyshev inequality.

\begin{cor}\label{cor:lip-image}
Let $X$ be an $\alpha$-Lipschitz image of a standard Gaussian vector in $\er^n$ and let $f\colon \er^n \to F$ be a function with bounded derivatives of order two. Assume moreover that $\Ex \nabla f(X) = 0$. Then for any $p \ge 2$,
\begin{align}
  \| f(X) - \Ex f(X)\|_p \le&
  C\alpha^2\sup_{z\in \er^n} \Bigg(\Ex\left\|\sum_{ij}\partial_{ij}f(z)(g_ig_j-\delta_{ij})\right\|
+\Ex\left\|\sum_{ i\neq j}\partial_{ij}f(z)g_{ij}\right\| \nonumber \\
  &+\sqrt{p}\sup_{\|x\|_2\leq 1}\Ex\left\|\sum_{i\neq j}\partial_{ij}f(z)x_ig_j\right\|\nonumber
\\
&+\sqrt{p}\sup_{\|(x_{ij})\|_2\leq 1}\left\|\sum_{ij}\partial_{ij}f(z)x_{ij}\right\|
+p\sup_{\|x\|_2\leq 1,\|y\|_2\leq 1}\left\|\sum_{ij}\partial_{ij}f(z)x_iy_j\right\|
\Bigg).\label{eq:bounded-Hessian}
\end{align}
In particular if $X$ is of mean zero, then
\begin{align}
\label{eq:lip-HW}
\left\|\sum_{ij}a_{ij}(X_iX_j-\Ex(X_iX_j))\right\|_p
\le&
C\alpha^2\Bigg(\Ex\left\|\sum_{ij}a_{ij}(g_ig_j-\delta_{ij})\right\|+\Ex\left\|\sum_{i\neq j}a_{ij}g_{ij}\right\|
\nonumber
\\
&+\sqrt{p}\sup_{\|x\|_2\leq 1}\Ex\left\|\sum_{i\neq j}a_{ij}x_ig_j\right\|
\nonumber+\sqrt{p}\sup_{\|(x_{ij})\|_2\leq 1}\left\|\sum_{ij}a_{ij}x_{ij}\right\|
\\
&+p\sup_{\|x\|_2\leq 1,\|y\|_2\leq 1}\left\|\sum_{ij}a_{ij}x_iy_j\right\|
\Bigg)
\end{align}
and the inequality \eqref{eq:HWnorm} is satisfied.

\end{cor}

\begin{proof}[Proof of Corollary \ref{cor:lip-image}]
Let $G_n = (g_1,\ldots,g_n)$, $G_n'=(g_1',\ldots,g_n')$ be independent standard Gaussian vectors in $\er^n$, independent of $X$.
  By an iterated application of \eqref{eq:Pisier} (the second time conditionally
 on $G_n$) we have
\begin{align*}
\Ex \|f(X) - \Ex f(X)\|^p
&\le C^p\alpha^p\Ex\|\langle \nabla f(X),G_n\rangle\|^p
\le C^{2p}\alpha^{2p} \Ex \left\|\sum_{ij} \partial_{ij} f(X) g_ig_j'\right\|^p
\\
& \le \tilde{C}^{2p}\alpha^{2p} \Ex \left\|\sum_{ij} \partial_{ij} f(X) (g_ig_j-\delta_{ij})\right\|^p ,
\end{align*}
where the last inequality follows by \cite[Theorem 2.2]{AG}.
To finish the proof of \eqref{eq:bounded-Hessian} it is now enough to apply Theorem \ref{thm:upper2d2} conditionally on $X$ and replace the expectation in $X$ by the supremum over $z \in \er^n$.

The inequality \eqref{eq:lip-HW} follows by a direct application of \eqref{eq:bounded-Hessian}.
\end{proof}

\section{Lower bounds}

In this part we show Proposition \ref{prop:lower2d} and the lower bound in Theorem \ref{thm:tails2d}.
We start with a simple lemma.

\begin{lem}
\label{lem:lower1}
Let $W= \|\sum_{i\neq j}a_{ij}g_ig_j\|_p+\|\sum_{i}a_{ii}(g_i^2-1)\|_p$. Then for any $p\geq 1$,
$$\frac{1}{3}W \leq \left\|\sum_{ij}a_{ij}(g_ig_j-\delta_{ij})\right\|_p \leq W.$$


\end{lem}

\begin{proof}
Let $(\ve_i)_{i}$ be a sequence of i.i.d. Rademacher random variables independent of $(g_i)_i$. We have by
symmetry of $g_i$ and Jensen's inequality,
\begin{align*}
\left\|\sum_{ij}a_{ij}(g_ig_j-\delta_{ij})\right\|_p&=
\left\|\sum_{ij}a_{ij}(\ve_i\ve_jg_ig_j-\delta_{ij})\right\|_p
\geq \left\|\Ex_\ve\sum_{ij}a_{ij}(\ve_i\ve_jg_ig_j-\delta_{ij})\right\|_p\\
&=\left\|\sum_{i}a_{ii}(g_i^2-1)\right\|_p.
\end{align*}
To conclude we use the triangle inequality in $L_p$ and get
\[
\left\|\sum_{i\neq j}a_{ij}g_ig_j\right\|_p\leq \left\|\sum_{ij}a_{ij}(g_ig_j-\delta_{ij})\right\|_p+
\left\|\sum_{i}a_{ii}(g_i^2-1)\right\|_p\leq 2\left\|\sum_{ij}a_{ij}(g_ig_j-\delta_{ij})\right\|_p.
\]

Adding the inequalities above yields the first estimate of the lemma. The second one follows trivially from the triangle inequality.
\end{proof}

\begin{proof}[Proof of Proposition \ref{prop:lower2d}]
Obviously
\[
\left\|\sum_{ij}a_{ij}(g_ig_j-\delta_{ij})\right\|_p\geq \Ex\left\|\sum_{ij}a_{ij}(g_ig_j-\delta_{ij})\right\|.
\]
Moreover, denoting by $\|\cdot\|_*$ the norm in the dual of $F$, we have
\begin{align*}
\left\|\sum_{ij}a_{ij}(g_ig_j-\delta_{ij})\right\|_p
&\geq \sup_{\|\varphi\|_*\leq 1}\left\|\sum_{ij}\varphi(a_{ij})(g_ig_j-\delta_{ij})\right\|_p
\\
&\geq \frac{1}{C}\left(\sqrt{p}\sup_{\|\varphi\|_*\leq 1}\|(\varphi(a_{ij}))_{ij}\|_{\mathrm{HS}}
+p\sup_{\|\varphi\|_*\leq 1}\|(\varphi(a_{ij}))_{ij}\|_{\mathrm{op}}\right)
\\
&=\frac{1}{C}\left(\sqrt{p}\sup_{\|(x_{ij})\|_2\leq 1}\left\|\sum_{ij}a_{ij}x_{ij}\right\|
+p\sup_{\|x\|_2\leq 1,\|y\|_2\leq 1}\left\|\sum_{ij}a_{ij}x_iy_j\right\|\right),
\end{align*}
where in the second inequality we used \eqref{mom}.

Lemma \ref{lem:lower1} and the decoupling inequality \eqref{eq:decoupling} yield
\begin{equation}\label{eq:decoupling-lower-bound}
\left\|\sum_{ij}a_{ij}(g_ig_j-\delta_{ij})\right\|_p
\geq\frac{1}{3}\left\|\sum_{i\neq j}a_{ij}g_ig_j\right\|_p
\geq \frac{1}{C}\left\|\sum_{i\neq j}a_{ij}g_ig_j'\right\|_p,
\end{equation}
where $(g_i')_{i}$ denotes an independent copy of $(g_i)_i$.

For any finite sequence $(b_i)_i$ in $(F,\|\ \cdot \ \|)$ we have
\begin{align}
\left\|\sum_{i}b_ig_i\right\|_p\geq \sup_{\|\varphi\|_*\leq 1}\left\|\sum_{i}\varphi(b_i)g_i\right\|_p
=\sup_{\|\varphi\|_*\leq 1}\|(\varphi(b_i))_i\|_2 \cdot \|g_1\|_p\geq
\frac{\sqrt{p}}{C}\sup_{\|x\|_2\leq 1}\left\|\sum_i x_ib_i\right\|. \label{ine:1}
\end{align}
Thus, by \eqref{eq:decoupling-lower-bound} and the Fubini Theorem, we get
\[
\left\|\sum_{ij}a_{ij}(g_ig_j-\delta_{ij})\right\|_p\geq
\frac{ \sqrt{p}}{C}\sup_{\|x\|_2\leq 1}\left\|\sum_{i\neq j}a_{ij}x_ig_j' \right\|_p
\geq \frac{\sqrt{p}}{C}\sup_{\|x\|_2\leq 1}\Ex\left\|\sum_{i\neq j}a_{ij}x_ig_j\right\|.
\]

\end{proof}

\section{Reduction to a bound on the supremum of a Gaussian process}

In this section we will reduce the upper estimates of Theorems \ref{thm:uppper2d1} and \ref{thm:upper2d2} to an estimate on expected value of a supremum of a certain Gaussian process. The arguments in this part of the article are well-known, we present them for the sake of completeness. In particular we will demonstrate the upper bounds given in \eqref{eq:Borel-Arcones-Gine}.

The first lemma shows that we may easily bound the diagonal terms.

\begin{lem}
\label{lem:diag}
For $p\geq 1$ we have
\begin{align*}
\left\|\sum_{i}a_{ii}(g_i^2-1)\right\|_p
&\leq C\left(\Ex\left\|\sum_{ij}a_{ij}(g_ig_j-\delta_{ij})\right\|
+\sqrt{p}\sup_{\|(x_{ij})\|_2\leq 1}\left\|\sum_{ij}a_{ij}x_{ij}\right\|\right.
\\
&\left.\phantom{aaaaaa}+p\sup_{\|x\|_2\leq 1,\|y\|_2\leq 1}\left\|\sum_{ij}a_{ij}x_iy_j\right\|\right).
\end{align*}
\end{lem}

\begin{proof}
Let $X_i$ be a sequence of i.i.d. standard symmetric exponential r.v's. A simple argument (cf. proof of Lemma 9.5 in \cite{AL}) shows that
\begin{align}
\label{ine:2}
\left\|\sum_{i}a_{ii}(g_i^2-1)\right\|_p\sim \left\|\sum_{i}a_{ii}g_i g'_i\right\|_p \sim \left\|\sum_{i}a_{ii}X_i\right\|_p.
\end{align}
 The latter quantity was bounded in \cite[Theorem 1]{LaSMlc}, thus
\begin{align*}
&\left\|\sum_{i}a_{ii}(g_i^2-1)\right\|_p
\sim
\left\|\sum_{i}a_{ii}(g_i^2-1)\right\|_1+\sqrt{p}\sup_{\|x\|_2\leq 1}\left\|\sum_{i}a_{ii}x_i\right\|
+p\sup_{i}\|a_{ii}\|
\\
&\leq C\left(\Ex\left\|\sum_{ij}a_{ij}(g_ig_j-\delta_{ij})\right\|+\sqrt{p}\sup_{\|(x_{ij})\|_2\leq 1}\left\|\sum_{ij}a_{ij}x_{ij}\right\|+p\sup_{\|x\|_2\leq 1,\|y\|_2\leq 1}\left\|\sum_{ij}a_{ij}x_iy_j\right\|\right),
\end{align*}
where in the last inequality we used Lemma \ref{lem:lower1}.
\end{proof}

From the next proposition it follows that in all our main results we can replace the term $\sup_{\|x\|_2\leq 1}\Ex\|\sum_{i \neq j}a_{ij}x_ig_j\|$ by $\sup_{\|x\|_2\leq 1}\Ex\|\sum_{ij}a_{ij}x_ig_j\|$.
\begin{prop}
\label{prop:diagest}
Under the assumption of Proposition \ref{prop:lower2d} we have for $p\geq 1$,
\begin{align*}
\sqrt{p}\sup_{\|x\|_2\leq 1}\Ex\left\|\sum_{i}a_{ ii}{x_ig_i}\right\|
\leq
&C\left(\Ex\left\|\sum_{ij}a_{ij}(g_{ij}-\delta_{ij})\right\|
+\sqrt{p}\sup_{\|(x_{ij})\|_2\leq 1}\left\|\sum_{ij}a_{ij}x_{ij}\right\|\right.
\\
&+\left. p\sup_{\|x\|_2\leq 1,\|y\|_2\leq 1}\left\|\sum_{ij}a_{ij}x_iy_j\right\|\right).
\end{align*}
\end{prop}

\begin{proof}
Let $(g_i')_i$ be an independent copy of the sequence $(g_i)_i$. Denoting by $\Ex'$ the expectation with respect to the variables $(g_i')_i$, we may estimate
\begin{align*}
\sqrt{p}\sup_{\|x\|_2\leq 1}\Ex\left\|\sum_{i}a_{ii}{x_ig_i}\right\|&\leq \Ex \sqrt{p}\sup_{\|x\|_2\leq 1} \left\|\sum_{i}a_{ii}{x_ig_i}\right\|
  \leq C\Ex \left(\Ex'\left \|\sum_{i}a_{ii} g_ig_i'\right\|^p\right)^{1/p} \\
  & \leq  C  \left\|\sum_i a_{ii} g_i g_i'\right\|_p \leq C \left\|\sum_i a_{ii} (g_i^2 - 1)\right\|_p,
\end{align*}
where the second inequality follows from \eqref{ine:1} applied conditionally on $(g_i)_i$, the third one from Jensen's inequality and the last one
 from \eqref{ine:2}. The assertion of the proposition follows now by Lemma \ref{lem:diag}.
\end{proof}

For the off-diagonal terms we use first the concentration approach.

\begin{prop}
\label{prop:red}
For $p\geq 1$ we have
\begin{multline*}
\left\|\sum_{i\neq j}a_{ij}g_ig_j\right\|_p\leq
C\Bigg(\Ex\left\|\sum_{i\neq j}a_{ij}g_ig_j\right\|
+\sqrt{p}\Ex\sup_{\|x\|_2\leq 1}\left\|\sum_{i\neq j}a_{ij}x_ig_j\right\|\\
+p\sup_{\|x\|_2\leq 1,\|y\|_2\leq 1}\left\|\sum_{ij}a_{ij}x_iy_j\right\|
\Bigg).
\end{multline*}
\end{prop}

\begin{proof}
Let
\begin{align*}
A:=\Bigg\{ z\in \er^n\colon\
&\left\|\sum_{i\neq j}a_{ij}z_iz_j\right\|\leq 4\Ex\left\|\sum_{i\neq j}a_{ij}g_ig_j\right\|,
\\
&\sup_{\|x\|_2\leq 1}\left\|\sum_{i\neq j}a_{ij}x_iz_j\right\|\leq
4\Ex \sup_{\|x\|_2\leq 1}\left\|\sum_{i\neq j}a_{ij}x_ig_j\right\|
\Bigg\}.
\end{align*}

Then $\gamma_{n}(A)\geq \frac{1}{2}$ by the Chebyshev inequality. Gaussian concentration gives
$\gamma_{n}(A+tB_{2}^{n})\geq 1-e^{-t^2/2}$ for $t\geq 0$. It is easy to check that for
$z\in A+tB_{2}^n$ we have
\[
\left\|\sum_{i\neq j}a_{ij}z_iz_j\right\|\leq 4S(t),
\]
where
\[
S(t)=\Ex\left\|\sum_{i\neq j}a_{ij}g_ig_j\right\|+2t\Ex \sup_{\|x\|_2\leq 1}\left\|\sum_{i\neq j}a_{ij}x_ig_j\right\|
+t^2\sup_{\|x\|_2\leq 1,\|y\|_2\leq 1}\left\|\sum_{i\neq j}a_{ij}x_iy_j\right\|.
\]
So
\[
\Pr\left(\left\|\sum_{i\neq j}a_{ij}g_ig_j\right\|> 4S(t)\right)\leq e^{-t^2/2}\quad \mbox{for }t\geq 0.
\]
Integrating by parts we get
$\|\sum_{i\neq j}a_{ij}g_ig_j\|_p\leq CS(\sqrt{p})$ for $p\geq 1$, which ends the proof, since
\begin{align}
\notag
\sup_{\|x\|_2\leq 1,\|y\|_2\leq 1}\left\|\sum_{i\neq j}a_{ij}x_iy_j\right\|
&\le \sup_{\|x\|_2\leq 1,\|y\|_2\leq 1}\left\|\sum_{i,j}a_{ij}x_iy_j\right\|+\max_{i}\|a_{ii}\|
\\
\label{eq:normopwithoutdiag}
&\le 2\sup_{\|x\|_2\leq 1,\|y\|_2\leq 1}\left\|\sum_{i,j}a_{ij}x_iy_j\right\|.
\end{align}
\end{proof}

For any symmetric matrix by \eqref{eq:decoupling} we have
$\Ex\|\sum_{i\neq j}a_{ij}g_ig_j\|\sim \Ex\|\sum_{i\neq j}a_{ij}g_ig_j'\|$. Moreover introducing decoupled
chaos enables us to release the assumptions of the symmetry of the matrix and zero diagonal.

Taking into account the above observations, Conjecture \ref{conj1_2d} reduces to the statement that for any $p\geq 1$ and any finite matrix $(a_{ij})$ in $(F,\|\ \cdot \ \|)$ we have
\begin{align}
\Ex\sup_{\|x\|_2\leq 1}\left\|\sum_{ij}a_{ij}g_ix_j\right\|
&\leq
C\Bigg(\frac{1}{\sqrt{p}}\Ex\left\|\sum_{ij}a_{ij}g_ig'_j\right\|
+\sup_{\|x\|_2\leq 1}\Ex\left\|\sum_{ij}a_{ij}g_ix_j\right\| \nonumber
\\
&\phantom{aaaaa}+\sup_{\|(x_{ij})\|_2\leq 1}\left\|\sum_{ij}a_{ij}x_{ij}\right\|
+\sqrt{p}\sup_{\|x\|_2\leq 1,\|y\|_2\leq 1}\left\|\sum_{ij}a_{ij}x_iy_j\right\|
\Bigg).\label{conj:supgauss0}
\end{align}

Let us rewrite \eqref{conj:supgauss0} in another language. We may assume that $F=\er^m$ for some
finite $m$ and $a_{ij}=(a_{ijk})_{k\leq m}$. Let $T=B_{F^*}$ be the unit ball in the dual space $F^*$.
Then \eqref{conj:supgauss0} takes the following form.

\begin{conj}
\label{conj:supgauss}

Let $p\geq 1$. Then for any triple indexed matrix $(a_{ijk})_{i,j\leq n, k\leq m}$ and bounded nonempty set $T\subset\er^m$ we have
\begin{align}\label{eq:corollary-supremum}
\Ex\sup_{\|x\|_2\leq 1, t\in T}\left|\sum_{ijk}a_{ijk}g_ix_jt_k\right|
&\leq
C\Bigg(\frac{1}{\sqrt{p}}\Ex\sup_{t\in T}\left|\sum_{ijk}a_{ijk}g_ig'_jt_k\right|
\nonumber\\
&\phantom{aa}
+\sup_{\|x\|_2\leq 1}\Ex\sup_{t\in T}\left|\sum_{ijk}a_{ijk}g_ix_jt_k\right|
+\sup_{t\in T}\left(\sum_{ij}\left(\sum_k a_{ijk}t_k\right)^2\right)^{1/2}
\nonumber\\
&\phantom{aa}+\sqrt{p}\sup_{\|x\|_2\leq 1,t\in T}\left(\sum_i\left(\sum_{jk}a_{ij}x_jt_k\right)^2\right)^{1/2}
\Bigg).
\end{align}
\end{conj}
Obviously it is enough to show this for finite sets $T$.

\section{Estimating suprema of Gaussian processes}
\label{sec:supgauss}

To estimate the supremum of a centered Gaussian process $(G_{v})_{v\in V}$ one needs to study the distance
on $V$ given by $d(v,v'):=(\Ex |G_v-G_{v'}|^2)^{1/2}$ (cf. \cite{Ta_book}). In the case of the Gaussian
process from Conjecture \ref{conj:supgauss} this distance is defined on $B_2^n\times T\subset \er^n\times \er^m$
by the formula
\[
d_A((x,t),(x',t')):=\left(\sum_{i}\left(\sum_{jk}a_{ijk}(x_jt_k-x'_jt'_k)\right)^2\right)^{1/2}
=\alpha_A(x\otimes t-x'\otimes t'),
\]
where $x\otimes t=(x_jt_k)_{j,k}\in \er^{nm}$ and $\alpha_A$ is a norm on $\er^{nm}$ given by
\[
\alpha_A(y):=\left(\sum_{i}\left(\sum_{jk}a_{ijk}y_{jk}\right)^2\right)^{1/2},
\]
(as in Conjecture \ref{conj:supgauss} in this section we do not assume that the matrix  $(a_{ijk})_{ijk}$ is symmetric or that it has
$0$ on the generalized diagonal).

Let
\[
B((x,t),d_A,r)=\left\{(x',t')\in \er^n\times T\colon\ \alpha_A(x\otimes t-x'\otimes t')\leq r\right\}
\]
be the closed ball in $d_A$ with center at $(x,t)$ and radius $r$.

Observe that
\[
\mathrm{diam}(B_2^n\times T,d_A)\sim \sup_{\|x\|_2\leq 1,t\in T}\left(\sum_i\left(\sum_{jk}a_{ijk}x_jt_k\right)^2\right)^{1/2}.
\]

Now we try to estimate entropy numbers $N(B_2^n\times T,d_A,\ve)$ for $\ve>0$  (recall that $N(S,\rho,\ve)$ is the smallest number of closed balls with  diameter $\ve$ in metric $\rho$ that cover the set $S$). To this end we
first introduce some notation. For a nonempty  bounded set $S$ in $\er^m$ let
\[
\beta_{A,S}(x):=\Ex\sup_{t\in S}\left|\sum_{ijk}a_{ijk}g_ix_jt_k\right|, \quad x\in\er^n.
\]

Observe that $\beta_{A,S}$ is a norm on $\er^n$. Moreover,
by the classical Sudakov minoration (\cite{Su} or \cite[Theorem 3.18]{LT}) for any $x\in \er^n$ and $\ve>0$ there exists a set $S_{x,\ve}\subset S$ of cardinality
at most $\exp(C\ve^{-2})$ such that
\[
\forall_{t\in  S}\ \exists_{t'\in S_{x,\ve}}\ \alpha_A(x\otimes (t-t'))\leq \ve \beta_{A,S}(x).
\]
For a finite set $S\subset \er^m$ and $\ve>0$ define a measure $\mu_{\ve,S}$ on $\er^n\times S$ in the following way
\[
\mu_{\ve,S}(C):=\int_{\er^n} \sum_{t\in S_{x,\ve}}\delta_{(x,t)}(C)d\gamma_{n,\ve}(x),
\]
where $\gamma_{n,\ve}$ is the distribution of the vector $\ve G_n$ (recall that $G_n$ is the standard Gaussian vector in $\er^n$).
Since $S$ is finite, we can choose sets $S_{x,\ve}$ in such a way that there are no problems with measurability.

To bound $N(B_2^n\times T,d_A,\ve)$ we need two lemmas.

\begin{lem}\cite[Lemma 1]{LaAoP}\label{lem1}
For any norms $\alpha_1,\alpha_2$ on $\er^n$, $y\in B_2^n$ and $\ve>0$,
\[
\gamma_{n,\ve}\left(x\colon\ \alpha_1(x-y)\leq 4\ve\Ex\alpha_1(G_n),\
\alpha_2(x)\leq 4\ve \Ex\alpha_2(G_n)+\alpha_2(y)\right)
\geq \frac{1}{2}\exp(-\ve^{-2}/2).
\]
\end{lem}

\begin{lem}
\label{lem:meas2d}
For any finite set $S$ in $\er^m$, any $(x,t)\in B_2^n\times S$ and $\ve>0$ we have
\[
\mu_{\ve,S} \left(B\left((x,t),d_A,r(\ve)\right)\right) \geq \frac{1}{2}\exp(-\ve^{-2}/2),
\]
where
\[
r(\ve)=r(A,S,x,t,\ve)
=4\ve^2\Ex\beta_{A,S}(G_n)+\ve\beta_{A,S}(x)+4\ve\Ex\alpha_A(G_n\otimes t).
\]
\end{lem}

\begin{proof}
Let
\[
U=\left\{x'\in \er^n\colon\ \beta_{A,S}(x')\leq 4\ve\Ex\beta_{A,S}(G_n)+\beta_{A,S}(x),
\alpha_A((x-x')\otimes t)\leq 4\ve \Ex\alpha_A(G_n\otimes t)\right\}.
\]
For any $x'\in U$ there exists $t'\in S_{x',\ve}$ such that
$\alpha_A(x'\otimes(t-t'))\leq \ve \beta_{A,S}(x')$. By the triangle inequality
\[
\alpha_A(x\otimes t-x'\otimes t')\leq \alpha_A((x-x')\otimes t)+\alpha_A(x'\otimes(t-t'))\leq r(\ve).
\]
Thus, by Lemma \ref{lem1}, $\mu_{\ve,S} \left(B\left((x,t),d_A,r(\ve)\right)\right) \geq\gamma_{n,\ve}(U)\geq \frac{1}{2}\exp(-\ve^{-2}/2)$.
\end{proof}

Having Lemma \ref{lem:meas2d} we can estimate the entropy numbers by a version of the usual volumetric argument.

\begin{cor}
\label{cor:entrestch}
For any $\ve>0$, $U\subset B_2^n$ and $S\subset \er^m$,
\begin{equation}
\label{eq:entrestch}
N\left(U\times S,d_A,8\ve^2\Ex\beta_{A,S}(G_n)
+2\ve\sup_{x\in U}\beta_{A,S}(x)+8\ve\sup_{t\in S}\Ex\alpha_A(G_n\otimes t)\right)
\leq \exp(C\ve^{-2})
\end{equation}
and for any $\delta>0$,

\begin{align*}
\sqrt{\log N(U\times S,d_A,\delta)}\leq
C\bigg(&\delta^{-1}\left(\sup_{x\in U}\beta_{A,S}(x)+\sup_{t\in S}\Ex\alpha_A(G_n\otimes t)\right)\\
&+\delta^{-1/2}(\Ex\beta_{A,S}(G_n))^{1/2}\bigg).
\end{align*}
\end{cor}

\begin{proof}
Let $r=4\ve^2\Ex\beta_{A,S}(G_n)+\ve\sup_{x\in U}\beta_{A,S}(x)+4\ve\sup_{t\in S}\Ex\alpha_A(G_n\otimes t)$ and $N=N(U\times S,d_A,2r)$. Then there exist points $(x_i,t_i)_{i=1}^N$ in $U\times S$ such that
$d_A((x_i,t_i),(x_j,t_j))> 2r$. To show \eqref{eq:entrestch} we consider two cases.

If $\ve>2$ then
\begin{align*}
2r
&\geq 4\sup_{x\in U}\beta_{A,S}(x)\geq 4\sup_{(x,t)\in U\times S}\Ex\left|\sum_{ijk}a_{ijk}g_it_jx_k\right|\\
&=4\sqrt{\frac{2}{\pi}}\sup_{(x,t)\in U\times S}\left(\sum_{i}\left(\sum_{jk}a_{ijk}t_jx_k\right)^2\right)^{1/2}\geq\diam(U\times S,d_A)
\end{align*}
so $N=1\leq \exp(C\ve^{-2})$.

If $\ve<2$, note that the balls $B((x_i,t_i),d_A,r)$ are disjoint and, by
Lemma \ref{lem:meas2d}, each of these balls has $\mu_{\ve,S}$ measure at least $\frac{1}{2}\exp(-\ve^{-2}/2)\geq\exp(-5\ve^{-2})$. On the other hand we obviously have $\mu_{\ve,S}(\er^n\times S)\leq \exp(C\ve^{-2})$. Comparing the upper and lower bounds on $\mu_{\ve,S}(\er^n\times S)$ gives \eqref{eq:entrestch} in this case.

The second estimate from the assertion is an obvious consequence of the first one.
\end{proof}

\begin{rem}\label{Dud}
The classical Dudley's bound on suprema of Gaussian processes (see e.g., \cite[Corollary 5.1.6]{delaPenaGine}) gives
\[
\Ex\sup_{\|x\|_2\leq 1, t\in T}\left|\sum_{ijk}a_{ijk}g_ix_jt_k\right|
\leq C\int_{0}^{\diam(B_2^n\times T,d_A)}\sqrt{\log N(B_2^n\times T,d_A,\delta)}d\delta.
\]
Observe that
\begin{align*}
\int_{0}^{\diam(B_2^n\times T,d_A)}\delta^{-1/2}(\Ex\beta_{A,T}(G_n))^{1/2}d\delta
&=2\sqrt{\diam(B_2^n\times T,d_A)\Ex\beta_{A,T}(G_n)}
\\
&\leq\frac{1}{\sqrt{p}}\Ex\beta_{A,T}(G_n)+\sqrt{p}\diam(B_2^n\times T,d_A)
\end{align*}
appears on the right hand side of \eqref{eq:corollary-supremum}. Unfortunately the other term in the estimate of
$\log^{1/2} N(B_2^n\times T,d_A,\delta)$ is not integrable. The remaining part of the proof is devoted to improve on Dudley's bound.
\end{rem}

We will now continue along the lines of \cite{LaAoP}.  We will need in particular to partition the set $T$ into smaller pieces $T_i$ such that $\sup_{t,s\in T_i}\Ex\alpha_A(G_n\otimes (t-t'))$ is small on each piece.
To this end we apply the following Sudakov-type estimate for chaoses, derived by Talagrand
(\cite{Ta_ch} or \cite[Section 8.2]{Ta_book}).

\begin{thm}
\label{thm:minchaos2d}
Let $\mathcal{A}$ be a subset of $n$ by $n$ real valued matrices and $d_2$, $d_{\infty}$ be distances associated to the Hilbert-Schmidt and operator norms respectively.
Then
\[
\ve\log^{1/4} N(\mathcal{A},d_2,\ve)\leq C\Ex\sup_{a\in \mathcal{A}}\sum_{ij}a_{ij}g_ig_j'
\quad \mbox{ for } \ve>0
\]
and
\[
\ve\log^{1/2} N(\mathcal{A},d_2,\ve)\leq C\Ex\sup_{a\in \mathcal{A}}\sum_{ij}a_{ij}g_ig_j'
\quad \mbox{ for } \ve>C\sqrt{\mathrm{diam}(\mathcal{A},d_{\infty})\Ex\sup_{a\in \mathcal{A}}\sum_{ij}a_{ij}g_ig_j'}.
\]
\end{thm}

To make the notation more compact let for $T\subset \er^m$ and $V\subset \er^n\times\er^m$,
\begin{align*}
s_A(T)&:=\Ex\beta_{A,T}(G_n)=\Ex\sup_{t\in T}\left|\sum_{ijk}a_{ijk}g_ig'_jt_k\right|,
\\
F_A(V)&:=\Ex\sup_{(x,t)\in V}\sum_{ijk}a_{ijk}g_ix_jt_k
\\
\Delta_{A,\infty}(T)&:=\sup_{\|x\|_2\leq 1,\|y\|_2\leq 1, t,t'\in T}
\left|\sum_{ijk}a_{ijk}x_iy_j(t_k-t'_k)\right|,
\\
\Delta_A(V)&:=\mathrm{diam}(V,d_A)=\sup_{(x,t),(x',t')\in V}\alpha_A(x\otimes t-x'\otimes t').
\end{align*}

\begin{cor}
\label{cor:decomp1}
Let $T$ be a subset of $\er^m$. Then for any $r>0$ there exists a decomposition
$T-T=\bigcup_{i=1}^N T_i$ such that, $N\leq e^{Cr}$ and
\[
\sup_{t,t'\in T_i}\Ex\alpha_A(G_n\otimes (t-t'))\leq
\min\left\{r^{-1/4}s_A(T),r^{-1/2}s_A(T)+C\sqrt{s_A(T)\Delta_{A,\infty}(T)}\right\}.
\]
\end{cor}

\begin{proof}
We use Theorem \ref{thm:minchaos2d} with $\mathcal{A}=\{(\sum_{k}a_{ijk}t_k)_{ij}\colon\ t\in T-T\}$. It is enough to observe that
\[
\Ex\sup_{b\in \mathcal{A}}\left|\sum_{ij}b_{ij}g_ig_j'\right|=s_A(T-T)\leq 2s_A(T), \quad
\mathrm{diam}(\mathcal{A},d_{\infty})= 2\Delta_{A,\infty}(T)
\]
and
\[
\Ex \alpha_A(G_n\otimes (t-t')) \leq  \left\|\left(\sum_k a_{ijk}(t_k-t_k')\right)_{ij}\right\|_{HS}.
\]
\end{proof}

On the other hand the dual Sudakov minoration (cf. formula (3.15) in \cite{LT}) yields the following

\begin{cor}
\label{cor:decomp2}
Let $U$ be a subset of $B_2^n$. Then for any $r>0$ there exists a decomposition
$U=\bigcup_{i=1}^N U_i$ such that $N\leq e^{Cr}$ and
\[
\sup_{x,x'\in U_i}\beta_{A,T}(x-x')\leq r^{-1/2}s_A(T).
\]
\end{cor}

Putting the above two corollaries together with Corollary \ref{cor:entrestch} we get the following decomposition of subsets of $B_2^n\times T$.

\begin{cor}
\label{cor:maindecomp}
Let $V\subset \er^n\times \er^m$ be such that $V-V\subset B_2^n\times (T-T)$. Then for $r\geq 1$ we may find a decomposition
$V=\bigcup_{i=1}^N((x_i,t_i)+V_i)$ such that $N\leq e^{Cr}$ and for each $1\leq i\leq N$,\\
i) $(x_i,t_i)\in V$, $V_i-V_i\subset V-V$, $V_i\subset B_2^n\times (T-T)$,\\
ii) $\sup_{(x,t)\in V_i}\beta_{A,T}(x)\leq r^{-1/2}s_A(T)$,\\
iii) $\sup_{(x,t)\in V_i}\Ex\alpha_A(G_n\otimes t)\leq
\min\left\{r^{-1/4}s_A(T),r^{-1/2}s_A(T)+C\sqrt{s_A(T)\Delta_{A,\infty}(T)}\right\}$,\\
iv) $\Delta_A(V_i)\leq \min\left\{r^{-3/4}s_A(T),r^{-1}s_A(T)+r^{-1/2}\sqrt{s_A(T)\Delta_{A,\infty}(T)}\right\}$.
\end{cor}

\begin{proof}
The assertion is invariant under translations  of the set $V$ thus we may assume that $(0,0)\in V$ and so $V\subset V-V\subset B_2^n\times (T-T)$.
By Corollaries \ref{cor:decomp1} and \ref{cor:decomp2} we may decompose $B_2^n=\bigcup_{i=1}^{N_1}U_i$, $T-T=\bigcup_{i=1}^{N_2} T_i$ in such a
way that $N_1,N_2\leq e^{Cr}$ and
\begin{align*}
\sup_{x,x'\in U_i}\beta_{A,T}(x-x')&\leq r^{-1/2}s_A(T),
\\
\sup_{t,t'\in T_i}\Ex\alpha_A(G_n\otimes (t-t'))&\leq\min\left\{r^{-1/4}s_A(T),r^{-1/2}s_A(T)+C\sqrt{s_A(T)\Delta_{A,\infty}(T)}\right\}.
\end{align*}
Let $V_{ij}:=V\cap (U_i\times T_j)$. If $V_{ij}\neq \emptyset$ we take any point
$(x_{ij},y_{ij})\in V_{ij}$ and using   Corollary \ref{cor:entrestch} with $\ve=r^{ -1/2}/C$ we decompose
\[
V_{ij}-(x_{ij},y_{ij})=\bigcup_{k=1}^{N_3}V_{ijk}
\]
in such a way that $N_3\leq e^{Cr}$ and
\begin{align*}
&\Delta_A(V_{ijk})
\\
&\leq \frac{1}{C} \Bigg(r^{-1}s_A(T)+r^{-1/2} \sup_{x'\in U_{i}} \beta_{A,T}(x'-x_{ij})
+r^{-1/2}\sup_{y'\in T_j} \Ex \alpha_A (G_n\otimes (y'-y_{ij})) \Bigg)
\\
&\leq \min\left\{r^{-3/4}{ s_A(T)},r^{-1}s_A(T)+r^{-1/2}\sqrt{s_A(T)\Delta_{A,\infty}(T)}\right\}.
\end{align*}
The final decomposition is obtained by relabeling of the decomposition $V = \bigcup_{ijk} ((x_{ij},y_{ij}) + V_{ijk})$.
\end{proof}

\begin{rem}
\label{rem:decomp}
We may also use a trivial bound in iii):
\[
\sup_{(x,t)\in V_i}\Ex\alpha_A(G_n\otimes t)\leq \sup_{t,t'\in T}\Ex\alpha_A(G_n\otimes (t-t'))
\leq 2\sup_{t\in T}\Ex\alpha_A(G_n\otimes t),
\]
this will lead to the following bound in iv):
\[
\Delta_A(V_i)\leq r^{-1}s_A(T)+r^{-1/2}\sup_{t\in T}\Ex\alpha_A(G_n\otimes t).
\]
\end{rem}

\begin{rem}
\label{sud:decomp}
By using Sudakov minoration instead of Theorem \ref{thm:minchaos2d} we may decompose the set $T=\bigcup_{i=1}^N T_i$, $N\leq \exp(Cr)$ in such a way that
$$\forall_{i \leq N} \sup_{t,t' \in T_i}\Ex\alpha_A(G_n\otimes (t-t'))\leq r^{-1/2}\Ex \sup_{t \in T} \sum_{ijk} a_{ijk}g_{ij}t_k.$$
This will lead to the following bounds in iii) and iv):
\begin{align*}
\sup_{(x,t)\in V_i}\Ex\alpha_A(G_n\otimes t) &\leq  r^{-1/2} \Ex \sup_{t \in T} \sum_{ijk} a_{ijk}g_{ij}t_k \\
\Delta_A(V_i)&\leq r^{-1} \left( \Ex \sup_{t \in T} \sum_{ijk} a_{ijk}g_{ij}t_k+s_A(T) \right).
\end{align*}
\end{rem}

\begin{lem}
\label{lem:translate}
Let $V$ be a subset of $B_2^n\times (T-T)$. Then for any $(y,s)\in \er^n\times \er^m$ we have
\[
F_A(V+(y,s))\leq F_A(V)+2\beta_{A,T}(y)+C\Ex\alpha_A(G_n\otimes s).
\]
\end{lem}

\begin{proof}
We have
\[
F_A(V+(y,s))
\leq F_A(V)+\Ex\sup_{(x,t)\in V}\sum_{ijk}a_{ijk}g_iy_jt_k
+\Ex\sup_{(x,t)\in V}\sum_{ijk}a_{ijk}g_ix_js_k.
\]
Obviously,
\[
\Ex\sup_{(x,t)\in V}\sum_{ijk}a_{ijk}g_iy_jt_k
\leq\Ex \sup_{t,t'\in T}\left|\sum_{ijk}a_{ijk}g_iy_j(t_k-t'_k)\right|\leq 2\beta_{A,T}(y).
\]
Moreover,
\begin{align*}
\Ex\sup_{(x,t)\in V}\sum_{ijk}a_{ijk}g_ix_js_k
& \leq \left(\Ex\sup_{x\in B_2^n}\left|\sum_{ijk}a_{ijk}g_ix_js_k\right|^2\right)^{1/2}=\left(\sum_{ij}\left(\sum_k a_{ijk}s_k\right)^2\right)^{1/2}\\
&=(\Ex\alpha_A(G_n\otimes s)^2)^{1/2}\leq C\Ex\alpha_A(G_n\otimes s),
\end{align*}
where in the second inequality we used the comparison of moments of Gaussian variables
\cite[Corollary 3.2]{LT}.
\end{proof}

\begin{prop}
\label{osz:main}
For any nonempty finite set $T$ in $\er^m$ and $p\geq 1$ we have
\begin{multline}
F_A(B_2^n\times T)\leq C\Bigg(\frac{\log (ep)}{\sqrt{p}}s_A(T)+\sup_{\|x\|_2\leq 1}\beta_{A,T}(x)\\
+\sup_{t\in T}\Ex\alpha_A(G_n\otimes t)+\log (ep)\sqrt{p}\Delta_A(B_2^n\times T)\Bigg), \label{wzor1}
\end{multline}
\begin{multline}
F_A(B_2^n\times T)\leq C\Bigg(\frac{1}{\sqrt{p}}s_A(T)+\sup_{\|x\|_2\leq 1}\beta_{A,T}(x)\\
+\log(ep)\sup_{t\in T}\Ex\alpha_A(G_n\otimes t)+\sqrt{p}\Delta_A(B_2^n\times T)\Bigg), \label{wzor2}
\end{multline}
\begin{multline}
F_A(B_2^n\times T)\leq C\Bigg(\frac{1}{\sqrt{p}}s_A(T)+\frac{1}{\sqrt{p}}\Ex \sup_{t \in T} \sum a_{ijk} g_{ij} t_k+\sup_{\|x\|_2\leq 1}\beta_{A,T}(x) \\
+\sup_{t\in T}\Ex\alpha_A(G_n\otimes t)+\sqrt{p}\Delta_A(B_2^n\times T)\Bigg). \label{wzor3}
\end{multline}
\end{prop}

\begin{proof}
First we prove \eqref{wzor1}
Let $l_0\in \ensuremath{\mathbb N}$ be such that $2^{l_0-1}\leq p< 2^{l_0}$. Define
\[
\Delta_0:=\Delta_A(B_2^n\times T),\quad \tilde{\Delta}_0:=\sup_{x\in B_2^n}\beta_{A,T}(x)
+\sup_{t\in T}\Ex\alpha_A(G_n\otimes t),
\]
\[
\Delta_l=2^{-3l/4}p^{-3/4}s_A(T),\quad \tilde{\Delta}_l=2^{-l/4}p^{-1/4}s_A(T), \quad l>l_0.
\]
and for $1\leq l\leq l_0$,
\begin{align*}
\Delta_l&:=2^{-l}p^{-1}s_A(T)+2^{-l/2}p^{-1/2}\sqrt{s_A(T)\Delta_{A,\infty}(T)},\\
\tilde{\Delta}_l&:=2^{-l/2}p^{-1/2}s_A(T)+C\sqrt{s_A(T)\Delta_{A,\infty}(T)}.
\end{align*}

Let for $l=0,1,\ldots$ and $m=1,2,\ldots$
\begin{align*}
c(l,m):=\sup\big\{F_A(V)\colon\
&V-V\subset B_2^n\times (T-T),\#V\leq m,
\\
&\Delta_A(V)\leq \Delta_l,
\sup_{(x,t)\in  V}(\beta_{A,T}(x)+\Ex\alpha_A(G_n\otimes t))\leq 2\tilde{\Delta}_l\big\}.
\end{align*}

Obviously $c(l,1)=0$. We will show that for $m>1$ and $l\geq 0$ we have
\begin{equation}
\label{eq:iteration}
c(l,m)\leq c(l+1,m-1)+C\left(2^{l/2}\sqrt{p}\Delta_l+\tilde{\Delta}_l\right).
\end{equation}
To this end take any set $V$ as in the definition of $c(l,m)$ and apply to it
Corollary \ref{cor:maindecomp} with $r=2^{l+1}p$ to obtain decomposition
$V=\bigcup_{i=1}^N ((x_i,t_i)+V_i)$. We may obviously assume that all $V_i$ have smaller cardinality than $V$. Conditions i)-iv) from Corollary  \ref{cor:maindecomp} easily imply that $F_A(V_i)\leq c(l+1,m-1)$.

Gaussian concentration (cf. \cite[Lemma 3]{LaAoP}) yields
\[
F_A(V)=F_A\left(\bigcup_{i}((x_i,t_i)+V_i)\right)\leq
C\sqrt{\log N}\Delta_A(V)+\max_{i}F_A((x_i,t_i)+V_i).
\]
Estimate \eqref{eq:iteration} follows since
\[
\sqrt{\log N}\Delta_A(V)\leq C2^{l/2}\sqrt{p}\Delta_l
\]
and for each $i$ by Lemma \ref{lem:translate} we have (recall that $(x_i,t_i)\in V$)
\[
F_A((x_i,t_i)+V_i)\leq F_A(V)+2\beta_{A,T}(x_i)+C\Ex\alpha_A(G_n\otimes t_i)\leq
c(l+1,m-1)+C\tilde{\Delta}_l.
\]

Hence
\begin{align*}
c(0,m)
&\leq C\left(\sum_{l=0}^{\infty} 2^{l/2}\sqrt{p}\Delta_l+\sum_{l=0}^\infty \tilde{\Delta}_l\right)
\\
&\leq C\left(\sqrt{p}\Delta_0+\tilde{\Delta}_0 +\frac{1}{\sqrt{p}}s_A(T)+l_0\sqrt{s_A(T)\Delta_{A,\infty}(T)}+
2^{-l_0/4}p^{-1/4}s_A(T)\right).
\end{align*}

Since $\log_2 p < l_0 \leq \log_2 p +1$ and
$\sqrt{s_A(T)\Delta_{A,\infty}(T)}\leq \frac{1}{\sqrt{p}}s_A(T)+\sqrt{p}\Delta_{A,\infty}(T)$  and clearly $\Delta_{A,\infty}(T)\leq \Delta_A(B_2^n\times T)$ we get for all $m \ge 1$,
\begin{multline*}
c(0,m)\leq  C\Bigg(\frac{\log (ep)}{\sqrt{p}}s_A(T)+\sup_{\|x\|_2\leq 1}\beta_{A,T}( x)\\
+\sup_{t\in T}\Ex\alpha_A(G_n\otimes t)+\log (ep)\sqrt{p}\Delta_A(B_2^n\times T)\Bigg).
\end{multline*}

To conclude the proof of \eqref{wzor1} it is enough to observe that
\[
F_A(B_2^n\times T)=2F_A\left(\frac{1}{2}B_2^n\times T\right)\leq 2\sup_{m\geq 1}c(0,m).
\]

The proofs of \eqref{wzor2} and \eqref{wzor3} are the same as the proof of \eqref{wzor1}. The only difference is that for
$1\leq l\leq l_0$  we  change the definitions of $\Delta_l$, $\tilde{\Delta}_l$ and we use Remarks \ref{rem:decomp} and \ref{sud:decomp} respectively. In the first case we take
\begin{align*}
\Delta_l&:=2^{-l}p^{-1}s_A(T)+2^{-l/2}p^{-1/2}\sup_{t \in T} \Ex \alpha_A(G_n \otimes t) \\
\tilde{\Delta}_l&:=2^{-l/2}p^{-1/2}s_A(T)+\sup_{t \in T} \Ex \alpha_A(G_n \otimes t),
\end{align*}
while in the second
\begin{align*}
\Delta_l:=2^{-l}p^{-1}\left(s_A(T)+\Ex \sup_{t\in T}\sum_{ijk} a_{ijk} g_{ij}t_k\right)\\
\tilde{\Delta}_l=2^{-l/2}p^{-1/2}\left( s_A(T)+\Ex \sup_{t\in T}\sum_{ijk} a_{ijk} g_{ij}t_k\right).
\end{align*}
\end{proof}

\begin{rem}
Note that the proof of Proposition \ref{osz:main} gives in fact the following estimate for any set $U \subset B_2^n\times T$. For any $p \ge 1$,
\begin{multline}
F_A(U)\leq C\Bigg(\frac{1}{\sqrt{p}}s_A(T)+\frac{1}{\sqrt{p}}\Ex \sup_{t \in T} \sum a_{ijk} g_{ij} t_k+\sup_{\|x\|_2\leq 1}\beta_{A,T}(x) \\
+\sup_{t\in T}\Ex\alpha_A(G_n\otimes t)+\sqrt{p}\Delta_A(U)\Bigg). \label{wzor-wzor}
\end{multline}
Indeed, as in the proof above one can reduce the problem to the case of $U \subset \frac{1}{2}B_2^n \times T$ and then it is enough to set $\Delta_0 = \Delta_A(U)$ and add the condition $V-V \subset U-U$ to the definition of $c(l,m)$.
\end{rem}

\section{Proofs of main results}\label{sec:proofs-of-main-results}

\begin{proof}[Proof of Theorems \ref{thm:uppper2d1} and \ref{thm:upper2d2}]
By Lemmas \ref{lem:lower1}, \ref{lem:diag} and Proposition \ref{prop:red} we need only to establish
\eqref{eq:upperest1}-\eqref{eq:upperest3} with $\|\sum_{ij}a_{ij}(g_ig_j-\delta_{ij})\|_p$ replaced by
$\sqrt{p}\Ex\sup_{\|x\|_2\leq 1}\|\sum_{i\neq j}a_{ij}g_ix_j\|$. We may assume that $F=\er^m$ and by \eqref{eq:normopwithoutdiag} also that $a_{ii}=0$, so
taking for $T$ the unit ball in the dual space $F^*$ we have

\[
\left\|\sum_{i\neq j}a_{ij}g_ix_j\right\|=\sup_{t\in T}\sum_{ijk}a_{ijk}g_ix_jt_k.
\]
Then, using the notation introduced in Section \ref{sec:supgauss},
\[
\Ex\sup_{\|x\|_2\leq 1}\left\|\sum_{i\neq j}a_{ij}g_ix_j\right\|=F_A(B_2^n\times T),\quad
\Ex\left\|\sum_{i\neq j}a_{ij}g_ix_j\right\|=\beta_{A,T}(x),
\]
\[
\sup_{\|(x_{ij})\|_2\leq 1}\left\|\sum_{ij}a_{ij}x_{ij}\right\|=\sup_{t\in T}(\Ex\alpha_A^2(G_n\otimes t))^{1/2}
\sim \sup_{t\in T}\Ex\alpha_A(G_n\otimes t),
\]
\[
\sup_{\|x\|_2\leq 1,\|y\|_2\leq 1}\left\|\sum_{ij}a_{ij}x_iy_j\right\|\sim \Delta_A(B_2^n\times T),
\]
\[
\Ex\left\|\sum_{i\neq j}a_{ij}g_{ij}\right\|=\Ex\sup_{t\in T}\sum_{ijk}a_{ijk}g_{ij}t_k\quad \mbox{and}\quad
\Ex\left\|\sum_{i\neq j}a_{ij}g_{i}g_j\right\|\sim s_A(T),
\]
where the last estimate follows by decoupling.
We conclude the proof invoking Proposition \ref{osz:main}.
\end{proof}

\begin{proof}[Proof of Theorem \ref{thm:tails2d}]
The proof of the lower bound \eqref{eq:thm6-lower-bound} will rely on two well known results, which we will state here in special cases, sufficient for our goals. The first one is the Paley-Zygmund inequality (see e.g., \cite[Corollary 3.3.2]{delaPenaGine}), which asserts that for any non-negative random variable $\xi$, such that $0 < \Ex \xi^2 < \infty$ and for any $\lambda \in [0,1]$,
\begin{displaymath}
  \Pr(\xi > \lambda \Ex \xi) \ge (1-\lambda)^2\frac{(\Ex \xi)^2}{\Ex \xi^2}.
\end{displaymath}

The other result we will need is a comparison of moments for Gaussian quadratic forms (see e.g. \cite[Theorem 3.2.10]{delaPenaGine}), where the case of general polynomials is considered), which states that if $S=\|\sum_{ij}a_{ij}(g_ig_j-\delta_{ij})\|$, then for any $q > p > 1$,
\begin{displaymath}
  \|S\|_{q} \le \frac{q-1}{p-1}\|S\|_p.
\end{displaymath}

Passing to the proof of \eqref{eq:thm6-lower-bound}, set $p=C\min\{t^2/U^2,t/V\}$. Note that it is enough to prove \eqref{eq:thm6-lower-bound} under the additional assumption that $p \ge 2$, since the left-hand side of this inequality is non-increasing with $t$, while the right-hand side is bounded from above, so the case $p < 2$ can be easily reduced to $p = 2$ at the cost of increasing the constant $C$.

By the Paley-Zygmund inequality (applied to $\xi = S^p$) and comparison of moments of Gaussian quadratic forms  we have for $p\geq 2$,
\[
\Pr\left(S\geq \frac{1}{2}(\Ex|S|^p)^{1/p}\right)=\Pr\left(|S|^p\geq \frac{1}{2^p}\Ex S^p\right)
\geq \left(1-\frac{1}{2^p}\right)^2\frac{(\Ex S^p)^2}{\Ex S^{2p}}\geq C_1^{-2p}.
\]
Thus \eqref{eq:thm6-lower-bound} follows by Proposition \ref{prop:lower2d}.

To derive the upper bound we  use Theorem \ref{thm:upper2d2}, estimate $\Pr(S\geq e\|S\|_p)\leq e^{-p}$ for $p\geq 1$ and make an analogous substitution.

\end{proof}

\begin{proof}[Proof of Theorem \ref{thm:hw}]
Recall that for $r > 0$ the $\psi_r$-norm of a random variable $Y$ is defined as
\begin{displaymath}
  \|Y\|_{\psi_r} = \inf\Big\{a>0\colon \Ex\exp\Big(\Big(\frac{|Y|}{a}\Big)^r\Big) \le 2\Big\}.
\end{displaymath}
(formally for $r < 1$ this is a quasi-norm, but it is customary to use the name  $\psi_r$-norm for all $r$).
By \cite[Lemma 5.4]{AW} if $k$ is a positive integer and $Y_1,\ldots,Y_n$ are symmetric random variables such that $\|Y\|_{\psi_{2/k}} \le M$, then
\begin{align}\label{eq:comparison-psi-r}
  \left\|\sum_{i=1}^n a_i Y_i\right\|_p \le C_k M \left\|\sum_{i=1}^n a_i g_{i1}\cdots g_{ik}\right\|_p,
\end{align}
where $g_{ik}$ are i.i.d. standard Gaussian variables (we remark that the lemma in \cite{AW} is stated only for $F = \er$ but its proof, based on contraction principle, works in any normed space).

To prove the theorem we will again establish a moment bound and then combine it with Chebyshev's inequality. Similarly as in the Gaussian setting we will treat the diagonal and off-diagonal part separately.
Let $\varepsilon_1,\ldots,\varepsilon_n$ be a sequence of i.i.d. Rademacher variables independent of $X_i$'s. For $p\ge 1$ we have
\begin{displaymath}
  \left\|\sum_i a_{ii}(X_i^2-\Ex X_i^2)\right\|_p \le 2 \left\|\sum_i a_{ii}\varepsilon_i X_i^2\right\|_p \le C \alpha^2 \left\|\sum_i a_{ii} g_{i1}g_{i2}\right\|_p,
\end{displaymath}
where in the first inequality we used symmetrization and in the second one \eqref{eq:comparison-psi-r} together with the observation $\|\varepsilon_i X_i^2\|_{\psi_1} \le C\alpha^2$ (which can be easily proved by integration by parts).

Now by \eqref{ine:2},
\begin{align*}
\left\|\sum_i a_{ii} g_{i1}g_{i2}\right\|_p \le C \left\|\sum_i a_{ii}(g_i^2-1) \right\|_p,
\end{align*}
and thus
\begin{align}\label{eq:diagonal-subgaussian}
 \left\|\sum_i a_{ii}(X_i^2-\Ex X_i^2)\right\|_p \leq C \alpha^2 \left\|\sum_i a_{ii}(g_i^2-1) \right\|_p.
\end{align}
The estimate of the off-diagonal part is analogous, the only additional ingredient is decoupling. Denoting $(X_i')_{i=1}^n$ an independent copy of the sequence $(X_i)_{i=1}^n$ and by $(\varepsilon_i)_{i=1}^n$, $(\varepsilon_i')_{i=1}^n$ (resp. $(g_i)_{i=1}^n$, $(g_i')_{i=1}^n$ ) independent sequences of Rademacher (resp. standard Gaussian) random variables, we have
\begin{align}
  \left\|\sum_{i\neq j} a_{ij} X_iX_j\right\|_p &\sim \left\|\sum_{i\neq j} a_{ij} X_iX_j'\right\|_p \sim \left\|\sum_{i\neq j} a_{ij} \varepsilon_iX_i\varepsilon_i'X_j'\right\|_p \nonumber \\
  &\le C\alpha^2\left\|\sum_{i\neq j} a_{ij} g_ig_j'\right\|_p \sim \alpha^2 \left\|\sum_{i\neq j} a_{ij} g_ig_j\right\|_p,\label{eq:off-diagonal-subgaussian}
\end{align}
where in the first and last inequality we used decoupling, the second one follows from iterated conditional application of symmetrization inequalities and the third one from iterated conditional application of \eqref{eq:comparison-psi-r} (note that by integration by parts we have $\|\varepsilon_i X_i\|_{\psi_2} \le C\alpha$).

Combining inequalities \eqref{eq:diagonal-subgaussian} and \eqref{eq:off-diagonal-subgaussian} with Lemma \ref{lem:lower1} we obtain
\begin{displaymath}
  \left\|\sum_{ij} a_{ij} (X_iX_j - \Ex X_iX_j)\right\|_p \le C\alpha^2 \left\|\sum_{ij} a_{ij}(g_ig_j - \delta_{ij})\right\|_p.
\end{displaymath}
To finish the proof of the theorem it is now enough to invoke moment estimates of Theorem \ref{thm:upper2d2} and use Chebyshev's inequality in $L_p$.
\end{proof}

\paragraph{Acknowledgement}  We would like to thank Mark Veraar for pointing us to references concerning Pisier’s contraction
property.

\end{document}